\documentclass[11pt,dvips,twoside,letterpaper]{article}
\usepackage{pslatex}
\usepackage{fancyhdr}
\usepackage{graphicx}
\usepackage{geometry}
 \RequirePackage[T1]{fontenc}

\def\figurename{Figure} 
\makeatletter
\renewcommand{\fnum@figure}[1]{\figurename~\thefigure.}
\makeatother

\def\tablename{Table} 
\makeatletter
\renewcommand{\fnum@table}[1]{\tablename~\thetable.}
\makeatother

\usepackage{amsmath}
\usepackage{amssymb}
\usepackage{amsfonts}
\usepackage{amsthm,amscd}

\newtheorem{theorem}{Theorem}[section]
\newtheorem{lemma}[theorem]{Lemma}
\newtheorem{corollary}[theorem]{Corollary}

\theoremstyle{definition}

\theoremstyle{remark}
\newtheorem{remark}[theorem]{Remark}

\numberwithin{equation}{section}

\def\P{\mathbb P}

\def\R{\mathbb R}
\def\E{\mathbb E}

\def\L{\mathbb L}

\def\E{\mathbb E}

\begin{document}
\title{\bf{A new type of reflected backward doubly stochastic differential equations}
\thanks{The work of Auguste Aman is support by TWAS Research Grants to individuals (No. 09-100 RG/MATHS/AF/AC-I-
UNESCO FR: 3240230311) and it been dedicated to all the dead of the post-election crisis in his country.\newline The work of Yong Ren is supported by the National Natural Science Foundation of  China (No 10901003), the Key
 Project of Chinese Ministry of Education (No 211077) and the Anhui Provincial Natural
 Science Foundation (No 10040606Q30).}}
\author{Auguste Aman$^1$\thanks{augusteaman5@yahoo.fr, corresponding author}\ \ and \ Yong Ren $^2$
\thanks{renyong@126.com and brightry@hotmail.com} \\
{\sl 1. U.F.R Math\'{e}matiques et Informatique, Universit\'{e} de Cocody, }\\{\it 582 Abidjan 22, C\^{o}te d'Ivoire}\\
{\sl 2. Department of Mathematics, Anhui Normal University,\;\;\;\;\;\;\;\;\;\:}\\ {\it Wuhu
241000, China}}

\date{}
\maketitle

\begin{abstract}
In this paper, we introduce a new kind of {\it "variant"} reflected backward doubly
stochastic differential equations (VRBDSDEs in short), where the drift is the nonlinear function of the barrier process. In the one stochastic case, this type of equations have been already studied by Ma and Wang \cite{MW}. They called it as {\it "variant"} reflected BSDEs (VRBSDEs in short) based on the general version of the Skorohod problem recently studied by Bank and El Karoui \cite{BK}. Among others, Ma and Wang \cite{MW} showed that VRBSDEs is a novel tool for some problems in finance and optimal stopping problems where no existing methods can be easily applicable. Since more of those models have their stochastic counterpart, it is very useful to transpose the work of Ma and Wang \cite{MW} to {\it doubly stochastic} version. In doing so, we firstly establish the stochastic variant Skorohod problem based on the stochastic representation theorem, which extends the work of Bank and El Karoui \cite{BK}.
We prove the existence and uniqueness of the solution for VRBDSDEs by means of the contraction mapping theorem.
By the way, we show the comparison theorem and stability result for the solutions of VRBDSDEs.
\end{abstract}


\vspace{.08in} \noindent \textbf{Keywords}: Reflected backward doubly stochastic differential equation, stochastic Skorohod problem, stochastic representation theorem.

\section{Introduction}
\setcounter{theorem}{0} \setcounter{equation}{0}
The theory of backward stochastic differential equations (BSDEs in
short) was developed by Pardoux and Peng \cite{PP90}. Precisely,
given a data $(\xi,f)$\ consisting of a progressively measurable
process $f$, so-called the generator, and a square integrable random
variable $\xi$, they proved the existence and uniqueness of an
adapted process $(Y,Z)$ solution to the following BSDEs:
$$Y_t=\xi+\int_t^T f(s,Y_s,Z_s)\,{\rm d}s-\int_t^TZ_s\,{\rm d}B_s,\quad 0\leq t\leq T.$$
These equations have attracted great interest due to their
connections with mathematical finance \cite{EPeQ,EQ}, stochastic
control and stochastic games \cite{EH,HL1,HL2}. Furthermore, it was
shown in various papers that BSDEs give the probabilistic
representation for the solution (at least in the viscosity sense) of
a large class of systems of semi-linear parabolic partial
differential equations (PDEs in short) \cite{Pa98,Pa99,PP92,Pe91}.

Further, other settings of BSDEs have been proposed. Especially, El-Karoui et
al. \cite{EKPPQ} have introduced the notion of reflected BSDEs
(RBSDEs in short), which is a BSDE but the solution is forced to
stay above a lower barrier. In details, a solution of such equations
is a triple of processes $(Y,Z,K)$ satisfying that
\begin{equation}\label{equation1}
Y_t=\xi+\int_t^Tf(s,Y_s,Z_s)\,{\rm d}s+K_T-K_t-\int_t^TZ_s\,{\rm d}B_s,\ Y_t\geq S_t,
\end{equation}
where $S$, so-called the barrier, is a given stochastic process. The
role of the continuous increasing process $K$ is to push the state
process upward with the minimal energy, in order to keep it above
$S$; in this sense, it satisfies $\int_0^T(Y_t-S_t)\,{\rm d}K_t=0.$
RBSDEs have been proven to be powerful tools in mathematical finance
\cite{CM,H}, the mixed game problems \cite{CK,HL3}, providing a
probabilistic formula for the viscosity solution of an obstacle
problem for a class of parabolic PDEs (\cite{CM1,EKPPQ,RX}) and so
on. On other interesting results on RBSDEs driven by a Brownian
motion with different barrier conditions, one can see Hamad\`ene
\cite{H1}, Lepeltier and Xu \cite{LX} and Peng and Xu \cite{PX}.

Very recently, Ma and Wang \cite{MW} introduced the so-called {\it Variant Reflected Backward
Stochastic Differential Equations} (VRBSDEs in short)
associated with the notion of variant Skorohod problem studied by Bank and El Karoui \cite{BK}, that is
\begin{equation}\label{equation2}
 Y_t=E\left\{X_T+\int_t^Tf(s,Y_s,A_s)ds|\mathcal{F}_t\right\}, 0\leq t\leq T,
 \end{equation}
where $X=\{X_t\}_{t\geq 0}$ is an optional process of class (D) and the solution $(Y,A)$ satisfies that
\begin{enumerate}
  \item [\rm(i)] $Y_t\leq X_t, 0\leq t\leq T, Y_T=X_T;$
  \item [\rm(ii)] $A=\{A_t\}$ is an adapted, increasing process such that $A_{0-}=-\infty$, and
  the following flat-off condition holds
  \begin{equation}\label{equation3} E\int_0^T|Y_t-X_t|dA_t=0.\end{equation}
\end{enumerate}
In addition, if the filtration $\mathcal{F}$ is generated by a Brownian motion $B$, then
(\ref{equation2}) has the following extension form
\begin{equation}\label{equation4} dY_t=-f(t,Y_t,Z_t,A_t)dt+Z_tdB_t, \ Y_t\leq X_t, 0\leq t \leq T, \ Y_T=X_T.\end{equation}
Unlike the role of $K$ in (\ref{equation1}) as stated previously, the process
$A$ here could be regarded as a density of a reflecting force, which acts through the
drift $f$ in a nonlinear manner. From the above statements, we can see that there is great
difference between the frameworks of RBSDEs  and VRBSDEs. Also,
even the fundamental well-posedness property of the VRBSDE cannot
be obtained by means of
the usual ways used in BSDE and RBSDE. This brand new kind of BSDEs has some important applications in finance
and optimal stopping problems (\cite{MW}).

In \cite{PardPeng}, Pardoux and Peng  proposed
another class of BSDEs, named backward doubly stochastic differential equations (BDSDEs in short) with the form:
\begin{equation}\label{equation5} Y_t=\xi+\int_t^Tf(s,Y_s,Z_s)ds+\int_t^Tg(s,Y_s,Z_s)dB_s
-\int_t^TZ_tdW_t,  0\leq t \leq T,\end{equation}
where the integral with respect to $\{B_t\}$ is a backward It\^{o} integral and
the integral with respect to $\{W_t\}$ is a standard forward It\^{o} integral. Those two types of integrals are particular cases of the Itô-Skorohod integral, see Nualart and Pardoux \cite{NP}.
Following it, some well-known works have been done in the probabilistic representation of certain quasi-linear
 stochastic partial differential equations by means of BDSDEs from different aspects, one can see
Bally and Matoussi \cite{MBa}, Boufoussi et al. \cite{B1,B2}, Buckdahn and Ma \cite{BM1,BM2,BM3}, Matoussi
and Scheutzow \cite{MS}, Zhang and
Zhao \cite{Zhang} and the references therein. Based on the reflected framework of El-Karoui et
al. \cite{EKPPQ}, Bahlali et al. \cite{bah}, Aman \cite{Aman} and Ren \cite{ren} respectively proved the
existence and uniqueness of the solution for a class of reflected BDSDEs (RBDSDEs in short) driven by Brownian motions and L\'{e}vy
processes. Especially, very recently, Matoussi and Stoica \cite{MS2} proved the existence and uniqueness result for the obstacle problem of quasi-linear parabolic stochastic PDEs by means of the
RBDSDEs.

Motivated by the aforementioned works, in this paper, we study a
class of variant reflected backward doubly stochastic differential equations (VRBDSDEs in short). In doing so, we firstly establish the stochastic variant Skorohod problem based on the stochastic representation theorem, which extends the work of Bank and El Karoui \cite{BK}.
We prove the existence and uniqueness of the solution for VRBDSDEs by means of the contraction mapping theorem.
In addition, we show the comparison theorem and the stability result for the solutions of VRBDSDEs.

Let us describe our plan. First, the formulation of the problems is proposed in Section 2. The main results are presented in Section 3.

\section{Formulation of the problems}
\setcounter{theorem}{0} \setcounter{equation}{0}
Let $(\Omega, \mathcal{F},\P)$ be a probability space and $T>0$ be fixed throughout this paper.  Let $\{W_t,\, 0\leq t\leq T\}$ and $\{B_t,\, 0\leq t\leq T\}$ be two mutually independent standard Brownian motion processes, with values respectively in $\R^d$ and in $\R^{\ell}$, define on $(\Omega, \mathcal{F},\P)$. Let $\mathcal{N}$ denote
the class of $\P$-null sets of $\mathcal{F}$. For each $t \in
[0,T]$, let us define
$$\mathcal{F}_{t}=\mathcal{F}_{t}^{W} \vee \mathcal{F}_{t,T}^{B},$$ where for any process $\{\eta_{t}; 0\leq t\leq T\},\; \mathcal{F}_{s,t}^{\eta}=\sigma
\{\eta_{r}-\eta_{s}; s\leq r \leq t \} \vee \mathcal{N}$ and, $\mathcal{F}_{t}^{\eta}=\mathcal{F}_{0,t}^{\eta}$.
\\
Knowing that $\{\mathcal{F}_{t}^W, t\in [0,T]\}$ is an increasing filtration and $\{\mathcal{F}_{t,T}^B, t\in [0,T]\}$ is a decreasing filtration, the collection $\{\mathcal{F}_{t}, t\in[0,T]\}$ is neither increasing nor decreasing so it does not constitute a filtration.

Let us describe the following spaces will frequently used in the sequel.
\begin{description}
\item $\bullet$ For any $n \in \mathbb{N}$,
$\mathcal{M}^{2}(0,T,\mathbb{R}^{n})$ denotes the set of (class of
$d\P\otimes dt$ a.e.) $n$-dimensional jointly measurable
random processes  $\{\varphi_{t}; 0\leq t\leq T \}$  such that
$$
\|\varphi \|_{\mathcal{M}^{2}}^{2}=\mathbb{E}\left(\int_{0}^{T}\mid \varphi_{t}
\mid^{2} dt\right)< +\infty.
$$
\item $\bullet$ $\mathcal{S}^{\infty}([0,T],\mathbb{R})$ denotes
the set of one dimensional continuous $\mathcal{F}_{t}$-measurable bounded random processes.
\item $\bullet$ $\L^{\infty}(\R)$ denotes the space of all $\mathcal{F}_{T}$-measurable bounded random variables.
\item $\bullet$ $\mathcal{M}_{0,T}$ denotes the space of all stopping times taking values in $[0,T]$.
\item $\bullet$ The process $X$ is said to belong to Class (D) on $[0,T]$ if the family of random variables $\{X_{\tau},\,\tau\in\mathcal{M}_{0,T}\}$ is uniformly integrable.
\end{description}
Next, let us give the standing assumptions relative to VRBDSDE.
\begin{description}
\item {(\bf A1)} The boundary processes $X=\{X_t, 0\leq t\leq T\}$ is assumed to be an optional process of class (D) and is lower-semi-continuous in expectation,
\item {(\bf A2)} The coefficients $f:[0,T]\times \Omega\times \R\times\R\rightarrow\R$ and $ g:[0,T]\times \Omega\times \R\rightarrow\R$ satisfy the following assumptions:
\begin{enumerate}
\item [\rm (i) ] for fixed $(\omega,t,y)\in\Omega\times[0,T]\times\R$, the function $f(t,\omega,y,\cdot)$ is continuous and strictly decreasing from $+\infty$ to $-\infty$;
\item [\rm (ii) ] for fixed $(y,l)\in\R^2$, the processes $f(\cdot,\cdot,y,l)$ and $g(\cdot,\cdot,y)$ are jointly measurable with
\begin{eqnarray*}
\E\int_{0}^{T}[|f(t,\omega,y,l)|+|g(t,\omega,y)|^2]dt< +\infty;
\end{eqnarray*}
\item [\rm (iii) ] there exists a constant $L>0$, such that for all fixed $t,\omega, l,$ it holds that
\begin{eqnarray*}
|f(t,\omega,y,l)-f(t,\omega,y',l)|&\leq& L|y-y'|,\\
|g(t,\omega,y)-g(t,\omega,y')|&\leq& L|y-y'|,\;\;\forall\,y,\, y'\in\R;
\end{eqnarray*}
\item [\rm (iv) ] there exist two constants $k>0$ and $K>0$, such that for all fixed $t,\omega, y,$ it holds that
\begin{eqnarray*}
k|l-l'|\leq |f(t,\omega,y,l)-f(t,\omega,y,l')|\leq K|l-l'|,\;\;\forall\,l,\, l'\in\R.
\end{eqnarray*}
\end{enumerate}
\end{description}
Given $\xi\in\L^{2}(\R)$ and the boundary process $X$, we consider the following VRBDSDE.
\begin{description}
\item [\rm (i) ]
\begin{eqnarray}
Y_t=\xi+\int_{t}^{T}f(s,Y_s,A_s)ds+\int_{t}^{T}g(s,Y_s)dB_{s}-\int_{t}^{T}Z_sdW_{s},\; 0\leq t\leq T;\label{eq11}
\end{eqnarray}
\item [\rm (ii) ] $ Y_t\leq X_t,\; 0\leq t\leq T,\;\; Y_T=X_T=\xi;$
\item [\rm (iii) ] the process $(A_t)_t$ is $\mathcal{F}_t$-measurable, increasing, c\`{a}dl\`{a}g (right continuous with left limits), and $A_{0^-}=-\infty$, such that $\displaystyle{\E\int_0^T|Y_t-X_t|dA_t=0}$.
\end{description}

The study of this new type of BDSDEs is based on the extension of {\it Stochastic Representation Theorem} initiated by Bank and El Karoui \cite{BK}. To do this, let us consider the following filtration $(\mathcal{G}_t)_{t\geq 0}$ defined by
\begin{eqnarray*}
\mathcal{G}_t=\mathcal{F}_{t}^{W} \vee \mathcal{F}_{T}^{B}.
\end{eqnarray*}
\begin{theorem}\label{T1}
Assume {\rm ({\bf A2})--(i),\, (ii)}. Then, every optional process $X$ of class (D) which is lower semi-continuous in expectation admits a representation of the form
\begin{eqnarray}
X_{S}=\E\left\{X_T+\int_{S}^{T}f\left(u,\sup_{S\leq v\leq u}L_v\right)du+\int_{S}^{T}g\left(u\right)dB_u|\mathcal{F}_S\right\}\label{eq00}
\end{eqnarray}
for any stopping times $S\in\mathcal{M}_{0,T}$, where $L$ is an optional process taking values in $\R\cup\{-\infty,+\infty\}$, and it can be characterized as follows
\begin{description}
\item [\rm $(i)$ ] $\, f\left(u,\sup_{S\leq v\leq u}L_v\right)\in L^1(\P\otimes dt),\,g\left(u\right)\in L^2(\P\otimes dt)$ for any stopping times $S$,
\item [\rm $(ii)$ ] $\, L_S=\mbox{ess}\inf_{\tau>S}l_{S,\tau}$, where the "$\mbox{ess}\inf$" is taken over all stopping times $S\in\mathcal{M}_{0,T}$ such that $S<T$, a.s.; and $l_{S,\tau}$ is the unique $\mathcal{F}_S$-measurable random variable satisfying that
\begin{eqnarray*}
\E\{X_{S}-X_{\tau}|\mathcal{F}_S\}=\E\left\{\int_{S}^{\tau}f\left(u,l_{S,\tau}\right)du+\int_{S}^{\tau}g\left(u\right)dB_u|\mathcal{F}_S\right\},
\end{eqnarray*}
\item [\rm $(iii)$ ] if $\;V(t,l)=\mbox{ess}\inf_{\tau\geq t}\E\left\{X_{\tau}+\int_{t}^{\tau}f\left(u,l\right)du+\int_{t}^{\tau}g(u)dB_u|\mathcal{F}_t\right\},\, t\in[0,T]$, is the value functions of a family of optimal stopping problems indexed by $l\in\R$, then
\begin{eqnarray*}
L_t=\sup\{l: V(t,l)=X_t\},\;\;\; t\in[0,T].
\end{eqnarray*}
\end{description}
\end{theorem}
\begin{proof}
Let $X$ be a optional process $X$ of class (D) which is lower semi-continuous in expectation and $g$ be a function given above. Setting
\begin{eqnarray*}
\widetilde{X}_t=X_t+\int^{t}_{0}g(u)dB_u,
\end{eqnarray*}
according to assumption $({\bf A1})$ and $({\bf A2})$, it is clear that $\widetilde{X}$ is a optional process of class (D) and is lower semi-continuous. Therefore, it follows from Theorem 3 in \cite{BK} that there exists an optional process $L$ taking values in $\R\cup\{-\infty,+\infty\}$ such that for any stopping times $S\in\mathcal{M}_{0,T}$,
\begin{eqnarray}
\widetilde{X}_{S}=\E\left\{\widetilde{X}_T+\int_{S}^{T}f\left(u,\sup_{S\leq v\leq u}L_v\right)du|\mathcal{G}_S\right\}.\label{eq000}
\end{eqnarray}
Moreover, $L$ is characterized as follows:
\begin{itemize}
\item $\, f\left(u,\sup_{S\leq v\leq u}L_v\right)\in L^1(\P\otimes dt)$ for any stopping times $S$, which satisfy $(i)$.
\item $\, L_S=\mbox{ess}\inf_{\tau>S}l_{S,\tau}$, where the "$\mbox{ess}\inf$" is taken over all stopping times $S\in\mathcal{M}_{0,T}$ such that $S<T$, a.s.; and $l_{S,\tau}$ is the unique $\mathcal{G}_S$-measurable random variable satisfying that
\begin{eqnarray}
\E\{\widetilde{X}_{S}-\widetilde{X}_{\tau}|\mathcal{G}_S\}=\E\left\{\int_{S}^{\tau}f\left(u,l_{S,\tau}\right)du|
\mathcal{G}_S\right\}.\label{eq001}
\end{eqnarray}
\item If $\widetilde{V}(t,l)=\mbox{ess}\inf_{\tau\geq t}\E\left\{\widetilde{X}_{\tau}+\int_{t}^{\tau}f\left(u,l\right)du|\mathcal{G}_t\right\},\, t\in[0,T]$, is the value functions of a family of optimal stopping problems indexed by $l\in\R$, then
\begin{eqnarray*}
L_t=\sup\{l:\widetilde{ V}(t,l)=\widetilde{X}_t\},\;\;\; t\in[0,T].
\end{eqnarray*}
\end{itemize}
Since $\widetilde{X}_S$ is $\mathcal{F}_S$-measurable and $\mathcal{F}_S\subset\mathcal{G}_S$, and according to the definition of $\widetilde{X}$, it follows from equalities \eqref{eq000} and \eqref{eq001} that
\begin{eqnarray*}
X_{S}=\E\left\{X_T+\int_{S}^{T}f\left(u,\sup_{S\leq v\leq u}L_v\right)du+\int^{T}_{S}g(u)dB_u|\mathcal{F}_S\right\}
\end{eqnarray*}
and
\begin{eqnarray}
\E\{X_{S}-X_{\tau}|\mathcal{F}_S\}=\E\left\{\int_{S}^{\tau}f\left(u,l_{S,\tau}\right)du+\int^{\tau}_{S}g(u)dB_u|
\mathcal{F}_S\right\},\label{eq002}
\end{eqnarray}
respectively. To prove $(ii)$, it remains to show that $l_{S,\tau}$ is a $\mathcal{F}_S$-measurable random variable, which is clear by \eqref{eq001}. To end the proof let us show $(iii)$.
In fact, equalities \eqref{eq001} and \eqref{eq002} provide

$\displaystyle\E\left\{X_{\tau}+\int_{S}^{\tau}f\left(u,l_{S,\tau}\right)du+\int^{\tau}_{S}g(u)dB_u|
\mathcal{G}_S\right\}$
$$=\E\left\{X_{\tau}+\int_{S}^{\tau}f\left(u,l_{S,\tau}\right)du+\int^{\tau}_{S}g(u)dB_u|
\mathcal{F}_S\right\}.$$
Hence, denoting
\begin{eqnarray*}
V(t,l)=\widetilde{V}(t,l)-\int^{\tau}_{t}g(u)dB_u,
\end{eqnarray*}
we have
\begin{eqnarray*}
V(t,l)=\mbox{ess}\inf_{\tau\geq t}\E\left\{X_{\tau}+\int_{t}^{\tau}f\left(u,l\right)du+\int^{\tau}_{S}g(u)dB_u|\mathcal{F}_t\right\}
\end{eqnarray*}
and
\begin{eqnarray*}
L_t=\sup\{l:\, V(t,l)=X_t\},\;\;\; t\in[0,T],
\end{eqnarray*}
which prove $(iii)$.
\end{proof}
A direct consequence of the previous stochastic representation theorem is the following stochastic variant Skorohod problem.
\begin{theorem}\label{T2}
Assume {\rm ({\bf A2})--(i),\, (ii)}. Then, for every optional process $X$ of class (D) which is lower semi-continuous in expectation, there exists a unique pair of
$\mathcal{F}_t$-measurable processes $(Y,A)$, where $Y$ is continuous and $A$ is increasing such that
\begin{eqnarray*}
Y_{t}=\E\left\{X_T+\int_{t}^{T}f\left(u,A_u\right)du+\int_{t}^{T}g\left(u\right)dB_u|\mathcal{F}_t\right\}, \,\, t\in[0,T].
\end{eqnarray*}
Furthermore, the process $A$ can be expressed as $A_t=\sup_{0\leq s\leq t^+}L_s$, where $L$ is the process in Theorem \ref{T1}.
\end{theorem}
Before give the proof of the above theorem, let us give a remark.
\begin{remark}
The previous theorem can be enounced as follows:
there exists a unique pair of $\mathcal{F}_t$-measurable processes $(Y,Z,A)$, where $Y$ is continuous and $A$ is increasing such that
\begin{eqnarray*}
Y_{t}=X_T+\int_{t}^{T}f\left(u,A_u\right)du+\int_{t}^{T}g\left(u\right)dB_u-\int^{T}_{t}Z_udW_u, \,\, t\in[0,T].
\end{eqnarray*}
\end{remark}
\begin{proof}
Let us define  $A_t=\sup_{0\leq s\leq t^+}L_s$, where $L$ is the process appears in \eqref{eq00} and the $\mathcal{G}_t$-square integrable martingale
\begin{eqnarray*}
M_t=\E\left\{X_T+\int_{0}^{T}f\left(u,A_u\right)du+\int_{0}^{T}g\left(u\right)dB_u|\mathcal{G}_t\right\},\;\; 0\leq t\leq T.
\end{eqnarray*}
An obvious extension of the It\^{o} martingale representation theorem yields the
existence of  a $\mathcal{G}_t$-progressively measurable process $\{Z_t\}$ with values in $\R^d$ such
that
\begin{eqnarray*}
&&\E\left(\int^T_0|Z_s|^2ds\right)<+\infty,\\
&&M_t=M_0+\int_0^tZ_sdW_s,\;\;\; 0\leq t\leq T.
\end{eqnarray*}
Hence,
\begin{eqnarray*}
M_T=M_t+\int_t^TZ_sdW_s,\;\;\; 0\leq t\leq T.
\end{eqnarray*}
Replacing $M_T$ and $M_t$, by their defining formulas and subtracting $\int_{0}^{t}f\left(u,A_u\right)du+\int_{0}^{t}g\left(u\right)dB_u$ from both
sides of the equality yields that
\begin{eqnarray*}
Y_{t}=X_T+\int_{t}^{T}f\left(u,A_u\right)du+\int_{t}^{T}g\left(u\right)dB_u-\int^{T}_{t}Z_udW_u,
\end{eqnarray*}
where
\begin{eqnarray}
Y_{t}=\E\left\{X_T+\int_{t}^{T}f\left(u,A_u\right)du+\int_{t}^{T}g\left(u\right)dB_u|\mathcal{G}_t\right\}.\label{eq00'}
\end{eqnarray}
It remains to show that $\{Y_t\}$ and $\{Z_t\}$ are $\mathcal{F}_t$-measurable. For $Y_t$, this is
obvious since for each $t$,
\begin{eqnarray*}
Y_{t}=\E\left\{\Theta|\mathcal{F}_t\vee\mathcal{F}^{B}_t\right\}.
\end{eqnarray*}
where $\Theta$ is $\mathcal{F}^{W}_T\vee\mathcal{F}^{B}_{t,T}$-measurable. Hence $\mathcal{F}^{B}_t$ is independent of $\mathcal{F}_t\vee\sigma(\Theta)$, and
\begin{eqnarray*}
Y_{t}=\E\left\{\Theta|\mathcal{F}_t\right\}.
\end{eqnarray*}
Now
\begin{eqnarray*}
\int^{T}_{t}Z_udW_u=X_T+\int_{t}^{T}f\left(u,A_u\right)du+\int_{t}^{T}g\left(u\right)dB_u-Y_{t},
\end{eqnarray*}
and the right side is $\mathcal{F}^{W}_T\vee\mathcal{F}^{B}_{t,T}$-measurable. Hence, from the It\^{o} martingale representation theorem $\{Z_s, t<s<T\}$
is $\mathcal{F}^{W}_s\vee\mathcal{F}^{B}_{t,T}$-adapted. Consequently, $Z_s$ is $\mathcal{F}^{W}_s\vee\mathcal{F}^{B}_{t,T}$-measurable, for any $t<s$ so it
is $\mathcal{F}^{W}_s\vee\mathcal{F}^{B}_{s,T}$ measurable. Therefore, the equality \eqref{eq00'} becomes
\begin{eqnarray*}
Y_{t}=\E\left\{X_T+\int_{t}^{T}f\left(u,A_u\right)du+\int_{t}^{T}g\left(u\right)dB_u|\mathcal{F}_t\right\},\label{eq00'}
\end{eqnarray*}
which shows the desired result.
\end{proof}

\section{Main results}
\setcounter{theorem}{0} \setcounter{equation}{0}
The main objective of this section is to prove the existence and uniqueness result to the new type of reflected BDSDEs. As mentioned in \cite{MW},
 we use the well-known contraction mapping theorem, to provide the existence and uniqueness of the solution.  Next, like as in \cite{MW}, we derive the comparison
 theorem and a stability result of such equations.
\subsection{Existence and uniqueness}
Let us make the following extra assumptions on the boundary process $X$ and the coefficients $f$ and $g$.
\begin{description}
\item ({\bf A3}) There exists a constant $\Gamma>0$, such that
\begin{enumerate}
\item [\rm (i)] for any $\mu\in\mathcal{M}_{0,T}$, it holds that
\begin{eqnarray*}
\mbox{ess}\sup_{\overset{\tau>\mu}{\tau\in\mathcal{M}_{0,T}}}\left\{\left|\frac{\E\left\{X_\tau-X_{\mu}|\mathcal{F}_{\mu}\right\}}
{\E\left\{\tau-\mu|\mathcal{F}_{\mu}\right\}}\right|
+\left|\frac{\left[\left(\E\int_{\mu}^{\tau}|g(u,0)|^2du|\mathcal{F}_\mu\right)\right]^{1/2}}{\E\left\{\tau-\mu|\mathcal{F}_{\mu}\right\}}\right|\right\}\leq \Gamma, \;\; a.s.;
\end{eqnarray*}
\item [\rm (ii)] $|f(t,0,0)|\leq \Gamma,\;\;\; t\in[0,T]$.
\end{enumerate}
\end{description}
Let us consider the following mapping $\Phi$ on $\mathcal{S}^{\infty}([0,T],\R)$: for a given process $y\in \mathcal{S}^{2}([0,T],\R)$, we define $\Phi(y)_t=Y_t,\, t\in[0,T]$, where $(Y,Z,A)$ is the unique solution of the variant Skorohod problem:
\begin{eqnarray}
&&Y_{t}=\xi+\int_{t}^{T}f\left(u,y_u,A_u\right)du+\int_{t}^{T}g\left(u,y_u\right)dB_u-\int^{T}_{t}Z_udW_u, \,\, t\in[0,T],\nonumber\\\label{eq2}\\
&&\E\int^T_0|Y_t-X_t|dA_t=0.\nonumber
\end{eqnarray}
It follows from Theorem \ref{T1} and Theorem \ref{T2} that the reflecting process $A$ is exactly determined by $y$ in this sense: $A_t=\sup_{0\leq s\leq t^+}L_s$ and
$L$ satisfies the stochastic representation:
\begin{eqnarray*}
X_{t}=\xi+\int_{t}^{T}f\left(u,y_u,\sup_{t\leq v\leq u}L_v\right)du+\int_{t}^{T}g\left(u,y_u\right)dB_u-\int_{t}^{T}\overline{Z}dW_u,\;\; t\in[0,T].
\end{eqnarray*}
Our goal is to prove that the mapping $\Phi$ is a contraction from $\mathcal{S}^{2}([0,T],\R)$ to itself. However, it should be noted that the contraction can only show
the existence and uniqueness of $Y$; the uniqueness of $A$ must be established separately.

We now derive some priori estimates that will be useful in the sequel. To begin with, let us consider the stochastic representation
\begin{eqnarray*}
X_{t}=\xi+\int_{t}^{T}f\left(u,0,\sup_{t\leq v\leq u}L^0_v\right)du+\int_{t}^{T}g\left(u,0\right)dB_u-\int^{T}_{t}Z^{0}_udW_u.
\end{eqnarray*}
Let us denote $A^0_t=\sup_{0\leq s\leq t^+}L_s^0$. Then we have the following result.
\begin{lemma}
Assume $({\bf A1}),\,({\bf A2})$ and $({\bf A3})$  hold. Then, it holds that
\begin{eqnarray}
\|A^0\|_{\infty}\leq \frac{3\sqrt{3}\Gamma}{k},
\end{eqnarray}
where $k$ and $\Gamma$ are the constants appearing in the previous assumptions.
\end{lemma}
\begin{proof}
For fixed $s\in[0,T]$ and any stopping times $\tau>s$, let us denote by $l^0_{s,\tau}$ the $\mathcal{F}_s$-measurable random variable such that
\begin{eqnarray*}
X_s-X_{\tau}=\int_{s}^{\tau}f\left(u,0,l^0_{s,\tau}\right)du+\int_{s}^{\tau}g\left(u,0\right)dB_u-\int^{\tau}_{s}Z^{0}_udW_u.
\end{eqnarray*}
Then, it follows from Theorem \ref{T1} that $L^0_s=\mbox{ess}\inf_{\tau>s}l_{s,\tau}^0$ and $A_t^0=\sup_{0\leq s\leq t^+}L^0_s$. On the other hand, we have
$$\E(X_s-X_{\tau}|\mathcal{F}_{s})-\E\left(\int_{s}^{\tau}f\left(u,0,0\right)du|\mathcal{F}_s\right)
-\E\left(\int_{s}^{\tau}g\left(u,0\right)dB_u|\mathcal{F}_s\right)$$
\begin{eqnarray}
=\E\left(\int_{s}^{\tau}[f\left(u,0,l^0_{s,\tau}\right)-f\left(u,0,0\right)]du|\mathcal{F}_s\right).\label{eq1}
\end{eqnarray}
On the set $\{\omega,\,l_{s,\tau}^0(\omega)<0\}$, since $f(t,0,\cdot)$ is decreasing and $l^0_{s,\tau}$ is $\mathcal{F}_s$-measurable, we have
\begin{eqnarray*}
\E\left(\int_{s}^{\tau}[f\left(u,0,l^0_{s,\tau}\right)-f\left(u,0,0\right)]du|\mathcal{F}_s\right)&\geq& \E\left(\int_{s}^{\tau}k|l^0_{s,\tau}|du|\mathcal{F}_s\right)\\
&\geq &k|l^0_{s,\tau}|\E\left(\tau-s|\mathcal{F}_s\right).
\end{eqnarray*}
According to \eqref{eq1}, we get
\begin{eqnarray*}
\E(X_s-X_{\tau}|\mathcal{F}_{s})-\E\left(\int_{s}^{\tau}f\left(u,0,0\right)du
|\mathcal{F}_s\right)-\E\left(\int_{s}^{\tau}g\left(u,0\right)dB_u|\mathcal{F}_s\right)
\geq k|l^0_{s,\tau}|\E\left(\tau-s|\mathcal{F}_s\right).
\end{eqnarray*}
In other words, on $\{l_{s,\tau}^0<0\}$, we have
\begin{eqnarray}
|l^0_{s,\tau}|&\leq&\frac{1}{k}\left\{\frac{\E(X_s-X_{\tau}|\mathcal{F}_{s})}{\E\left(\tau-s|\mathcal{F}_s\right)}-\frac{\E\left(\int_{s}^{\tau}f\left(u,0,0\right)du|\mathcal{F}_s\right)}
{\E\left(\tau-s|\mathcal{F}_s\right)}-\frac{\E\left(\int_{s}^{\tau}g\left(u,0\right)dB_u|\mathcal{F}_s\right)}{\E\left(\tau-s|\mathcal{F}_s\right)}\right\}.\;\;\;\;\;\;\;\; \;\;\;\;\;\;\;\;\;\; \label{eq2}
\end{eqnarray}
We can show similarly that on the set $\{l_{s,\tau}^0>0\}$ the following relation holds
\begin{eqnarray}
l^0_{s,\tau}&\leq&\frac{1}{k}\left\{-\frac{\E(X_s-X_{\tau}|\mathcal{F}_{s})}{\E\left(\tau-s|\mathcal{F}_s\right)}+\frac{\E\left(\int_{s}^{\tau}
f\left(u,0,0\right)du|\mathcal{F}_s\right)}
{\E\left(\tau-s|\mathcal{F}_s\right)}+\frac{\E\left(\int_{s}^{\tau}g\left(u,0\right)dB_u|\mathcal{F}_s\right)}{\E\left(\tau-s|
\mathcal{F}_s\right)}\right\}.\;\;\;\;\;\;\;\;\;\;\;\;\;\;\label{eq3}
\end{eqnarray}
Putting \eqref{eq2} and \eqref{eq3} together, we have
\begin{eqnarray}
|l^0_{s,\tau}|^2&\leq&\frac{3}{k^2}\Bigg\{\left|\frac{\E(X_s-X_{\tau}|\mathcal{F}_{s})}{\E\left(\tau-s|\mathcal{F}_s\right)}\right|^2
+\left|\frac{\E\left(\int_{s}^{\tau}|f\left(u,0,0\right)|du|\mathcal{F}_s\right)}
{\E\left(\tau-s|\mathcal{F}_s\right)}\right|^2\nonumber\\
&&+\frac{\E\left(\left|\int_{s}^{\tau}g\left(u,0\right)dB_u\right|^2
|\mathcal{F}_s\right)}{\left|\E\left(\tau-s|\mathcal{F}_s\right)\right|^2}\Bigg\}.\;\;\;\;\;\;\;\;\;\;\;\;\;\;\label{eq4}
\end{eqnarray}
Using conditional expectation version of isometry property, we get
\begin{eqnarray*}
\E\left(\left|\int_{s}^{\tau}g\left(u,0\right)dB_u\right|^2|\mathcal{F}_s\right)&=& \E\left(\int_{s}^{\tau}|g\left(u,0\right)|^2du|\mathcal{F}_s\right)
\end{eqnarray*}
which together with \eqref{eq4} leads to
\begin{eqnarray}
|l^0_{s,\tau}|&\leq&\frac{\sqrt{3}}{k}\Bigg\{\left|\frac{\E(X_s-X_{\tau}|\mathcal{F}_{s})}{\E\left(\tau-s|\mathcal{F}_s\right)}\right|
+\frac{\E\left(\int_{s}^{\tau}|f\left(u,0,0\right)|du|\mathcal{F}_s\right)}
{\E\left(\tau-s|\mathcal{F}_s\right)}\nonumber\\
&&+\frac{\left[\E\left(\int_{s}^{\tau}|g\left(u,0\right)|^2du
|\mathcal{F}_s\right)\right]^{1/2}}{\E\left(\tau-s|\mathcal{F}_s\right)}\Bigg\}.\;\;\;\;\;\;\;\;\;\;\;\;\;\;\label{eq5}
\end{eqnarray}
Since
\begin{eqnarray*}
|A^{0}_t|=\left|\sup_{0\leq s\leq t^+}L^0_s\right|\leq \sup_{0\leq s\leq t^+}|L^0_s|=\sup_{0\leq s\leq t^+}\left\{\mbox{ess}\inf_{\tau>s}|l_{s,\tau}|\right\},
\end{eqnarray*}
we derive from \eqref{eq4} and $({\bf A3})$ that
\begin{eqnarray*}
|A^{0}_t|\leq\sup_{0\leq s\leq t^+}\left\{\mbox{ess}\inf_{\tau>s}|l_{s,\tau}|\right\}\leq \frac{3\, \sqrt{3}\,\Gamma}{k}
\end{eqnarray*}
and ends the proof.
\end{proof}

\begin{lemma}\label{L312}
Assume $({\bf A1}),\,({\bf A2})$ and $({\bf A3})$ hold. Then, for any $t\in[0, T]$, it holds almost surely that
\begin{eqnarray*}
|A_t-A'_t|\leq \frac{\sqrt{2}\,L}{k}(1+\sqrt{T})\|y-y'\|_{\infty}.
\end{eqnarray*}
\end{lemma}
\begin{proof}
Again, we fix $s\in [0,T]$ and let $\tau\in \mathcal{M}(0,T)$ be such that $\tau>s$ a.s. Let us consider, according to Theorem 2.1, $l_{s,\tau},\, l'_{s,\tau}$ two $\mathcal{F}_s$-measurable random variables such that

$\displaystyle \E(X_s-X_{\tau}|\mathcal{F}_{s})=\E\left\{\int_{s}^{\tau}f\left(u,y_u,l_{s,\tau}\right)du
+\int_{s}^{\tau}g\left(u,y_u\right)dB_u|\mathcal{F}_s\right\}$
\begin{eqnarray}
=\E\left\{\int_{s}^{\tau}f\left(u,y'_u,l'_{s,\tau}\right)du+\int_{s}^{\tau}g\left(u,y'_u\right)dB_u
|\mathcal{F}_s\right\}.\label{eq6}
\end{eqnarray}
Let us denote $D^{\tau}_{s}=\left\{\omega/ l_{s,\tau}(\omega)>l'_{s,\tau}(\omega)\right\}$, thus $D^{\tau}_{s}\in \mathcal{F}_s$, for any stopping times $\tau>s$. Since ${\bf 1}_{D^{\tau}_{s}}$ is $\mathcal{F}_s$-measurable, it follows from \eqref{eq6} that
\\
$\displaystyle \left[\E\left(\int_{s}^{\tau}|f\left(u,y_u,l_{s,\tau}\right)-f\left(u,y_u,l'_{s,\tau}\right)|{\bf 1}_{D^{\tau}_{s}}du|\mathcal{F}_s\right)\right]^2$
\begin{eqnarray}
=\left[\E\left(\int_{s}^{\tau}[f\left(u,y'_u,l'_{s,\tau}\right)-f\left(u,y_u,l'_{s,\tau}\right)]{\bf 1}_{D^{\tau}_{s}}du
+\int_{s}^{\tau}(g\left(u,y'_u\right)-g\left(u,y_u\right)){\bf 1}_{D^{\tau}_{s}}dB_u|\mathcal{F}_s\right)\right]^2.\nonumber\\\label{eq60}
\end{eqnarray}
By assumption $({\bf A2})$-$(iv)$, we have
\begin{eqnarray}
\left[\E\left(\int_{s}^{\tau}|f\left(u,y_u,l_{s,\tau}\right)-f\left(u,y_u,l'_{s,\tau}\right)|{\bf 1}_{D^{\tau}_{s}}du|\mathcal{F}_s\right)\right]^2
\geq k^2|l_{s,\tau}-l'_{s,\tau}|^2[\E\left\{\tau-s|\mathcal{F}_s\right\}{\bf 1}_{D^{\tau}_{s}}]^2.\nonumber\\\label{eq6'}
\end{eqnarray}
Next, assumption ({\bf A2})-(iii) together with conditional expectation version of isometry
 property lead to
 \\
 $\displaystyle\left[\E\left(\int_{s}^{\tau}[f\left(u,y'_u,l'_{s,\tau}\right)-f\left(u,y_u,l'_{s,\tau}\right)]{\bf 1}_{D^{\tau}_{s}}du+\int_{s}^{\tau}(g\left(u,y'_u\right)-g\left(u,y_u\right)){\bf 1}_{D^{\tau}_{s}}dB_u|\mathcal{F}_s\right)\right]^2$
\begin{eqnarray}
&\leq & 2\left[\E\left(\int_{s}^{\tau}|f\left(u,y_u,l'_{s,\tau}\right)-f\left(u,y'_u,l'_{s,\tau}\right)|{\bf 1}_{D^{\tau}_{s}}du\right)\right]^2\nonumber\\
&&+2\left[\E\int_{s}^{\tau}|g\left(u,y'_u\right)-g\left(u,y_u\right)|^2{\bf 1}_{D^{\tau}_{s}}du|\mathcal{F}_s\right]\nonumber\\
&\leq & 2L^2\|y-y'\|^2_{\infty}[\E\left\{\tau-s|\mathcal{F}_s\right\}{\bf 1}_{D^{\tau}_{s}}]^2+2L^2\|y-y'\|^2_{\infty}\E\left(\tau-s|\mathcal{F}_s\right){\bf 1}_{D^{\tau}_{s}}.\label{eq6''}
\end{eqnarray}
Combining \eqref{eq6'} and \eqref{eq6''} with \eqref{eq60}, we obtain
\begin{eqnarray*}
k|l_{s,\tau}-l'_{s,\tau}|\E\left\{\tau-s|\mathcal{F}_s\right\}\leq \sqrt{2}L\|y-y'\|_{\infty}\E\left\{\tau-s|\mathcal{F}_s\right\}+\sqrt{2}L\|y-y'\|_{\infty}
[\E\left(\tau-s|\mathcal{F}_s\right)]^{1/2},
\end{eqnarray*}
on $D^{\tau}_{s}$. Thus,
\begin{eqnarray*}
|l_{s,\tau}-l'_{s,\tau}|\leq \frac{\sqrt{2}\,L}{k}(1+[\E\left\{(\tau-s)|\mathcal{F}_s\right\}]^{-1/2})\|y-y'\|_{\infty}
\end{eqnarray*}
on $D^{\tau}_{s}$, since $\tau>s$. Similarly, we can show that the inequality holds on the complement of $D^{\tau}_{s}$ as well. Therefore, we have
\begin{eqnarray}
|l_{s,\tau}-l'_{s,\tau}|\leq \frac{\sqrt{2}\,L}{k}(1+[\E\left\{\tau-s|\mathcal{F}_s\right\}]^{-1/2})\|y-y'\|_{\infty}.\label{eq7}
\end{eqnarray}
Next, since $L_s=\mbox{ess}\inf_{\tau>s}l_{s,\tau},\; L'_s=\mbox{ess}\inf_{\tau>s}l'_{s,\tau},\; A_t=\sup_{0\leq s\leq t}L_{s}$ and $A'_t=\sup_{0\leq s\leq t}L'_{s}$, we conclude from \eqref{eq7} that
\begin{eqnarray*}
|A_t-A'_t|=\left|\sup_{0\leq s\leq t}L_{s}-\sup_{0\leq s\leq t}L'_{s}\right|&\leq& \sup_{0\leq s\leq t}\left|\mbox{ess}\inf_{\tau>s}l_{s,\tau}-\mbox{ess}\inf_{\tau>s}l'_{s,\tau}\right|\\
&\leq&\sup_{0\leq s\leq t}\mbox{ess}\sup_{\tau>s}|l_{s,\tau}-l'_{s,\tau}|\\
&\leq &\sup_{0\leq s\leq t}\mbox{ess}\sup_{\tau>s}\frac{\sqrt{2}L}{k}\left[1+\left(\E\left\{(\tau-s)|\mathcal{F}_s\right\}\right)^{-1/2}\right]\|y-y'\|_{\infty}\\
&\leq & \frac{\sqrt{2}L}{k}(1+\sqrt{T})\|y-y'\|_{\infty}.
\end{eqnarray*}
\end{proof}
We are now ready to prove the main result of this section, the existence and uniqueness
of the solution to the VRBDSDE.
\begin{theorem}\label{T311}
Assume $({\bf A1}),\,({\bf A2})$ and $({\bf A3})$ hold. Assume further that $$ {2TL\left(1+\sqrt{2}\frac{K}{k}\left(1+\sqrt{T}\right)\right)+L\sqrt{2T}< 1},$$
then the VRBDSDE \eqref{eq11} admits a unique solution $(Y,A)$.
\end{theorem}
\begin{proof}
First, let us show that the mapping $\Phi$ defined by \eqref{eq2} is from $\mathcal{S}^{\infty}$ to itself. To do this, we note that by using assumption $({\bf A1})$ and Lemmas 3.1 and 3.2, we derive
\begin{eqnarray}
|Y_t|^2=|\Phi(y)_t|^2\leq 3\E\left\{|\xi|+\left(\int^{T}_{t}|f(s,y_s,A_s)|ds\right)^2
+\left(\int_t^Tg(s,y_s)dB_s\right)^2|\mathcal{F}_t\right\}.\label{eq100}
\end{eqnarray}
We have

$\displaystyle \E\left\{\left(\int^{T}_{t}f(s,y_s,A_s)ds\right)^2|\mathcal{F}_t\right\}$
\begin{eqnarray}
&\leq& 4T^2\left(K^2\|A-A^0\|^2_{\infty}+L^2\|y\|^2_{\infty}+K^2\|A^0\|^2_{\infty}+\Gamma^2\right)\nonumber\\
&\leq& 4T^2L^2\left(1+2\frac{K^2}{k^2}\left(1+\sqrt{T}\right)\|y\|^2_{\infty}\right)+4T^2\left(1+27\frac{K^2}{k^2}\right)\Gamma^2\label{eq101}
\end{eqnarray}
and
\begin{eqnarray}
\E\left\{\left(\int^{T}_{t}g(s,y_s)|ds\right)^2|\mathcal{F}_t\right\}
&\leq& 2\E\left\{\left(\int_0^T|g(s,0)|^2ds\right)|\mathcal{F}_t\right\}+2L^2T\|y\|_{\infty}^2.\label{eq102}
\end{eqnarray}
It follows from \eqref{eq100}, \eqref{eq101} and \eqref{eq102} that
\begin{eqnarray*}
|Y_t|&\leq& \|\xi\|+\sqrt{2}\left[\E\left(\int_0^T|g(s,0)|^2ds|\mathcal{F}_t\right)\right]^{1/2}
+L\left[2T\left(1+\sqrt{2}\frac{K}{k}\left(1+\sqrt{T}\right)\right)+\sqrt{2T}\right]\|y\|_{\infty}\\
&&+2T\left(1+3\sqrt{3}\frac{K}{k}\right)\Gamma.
\end{eqnarray*}
As it is known by assumption that $\xi$ belongs to $\L^{\infty}$, we deduce from $({\bf A3})$-$(i)$ that $Y=\Phi(y)$ belongs to $S^{\infty}$.
Now, let us  prove that $\Phi$ is a contraction. For $y,\, y'\in S^{\infty}$, we denote $Y=\Phi(y)$ and $Y'=\Phi(y')$. Then, for $t\in[0,T]$, we have
\begin{eqnarray}\label{eq8}
|\Phi(y)-\Phi(y')|^2&=&\left|\E\left\{\int_t^T[f(s,y_s,A_s)-f(s,y'_s,A'_s)]ds+\int_t^T[g(s,y_s)-g(s,y'_s)]dB_s|\mathcal{F}_t\right\}\right|^2 \nonumber \\
&\leq&  2\left|\E\left\{\int_t^T|f(s,y_s,A_s)-f(s,y'_s,A'_s)|ds|\mathcal{F}_t\right\}\right|^2\nonumber \\&&+2\E\left\{\left|\int_t^T[g(s,y_s)-g(s,y'_s)]dB_s\right|^2|\mathcal{F}_t\right\}.
\end{eqnarray}
Applying assumption on $f$ and Lemma 3.2, we derive that
\begin{eqnarray}
\left|\E\left(\int_t^T|f(s,y_s,A_s)-f(s,y'_s,A'_s)|ds|\mathcal{F}_t\right)\right|^2&\leq& 2T^2\left(L^2\|y'-y\|^2_{\infty}+K^2\|A-A'\|_{\infty}^2\right)\nonumber\\
&\leq& 2T^2\left[L^2+K^2\frac{2L^2}{k^2}\left(1+\sqrt{T}\right)^2\right]\|y'-y\|_{\infty}.\nonumber\\\label{eq9}
\end{eqnarray}
Moreover, it follows from conditional expectation version of isometry property that
\begin{eqnarray}
\E\left\{\left|\int_t^T[g(s,y_s)-g(s,y'_s)]dB_s\right|^2|\mathcal{F}_t\right\}&=& \E\left\{\left(\int_t^T|g(s,y_s)-g(s,y'_s)|^2ds\right)|\mathcal{F}_t\right\}\nonumber\\
&\leq&L^2T\|y-y'\|^2_{\infty}\label{eq10}.
\end{eqnarray}
Finally, putting \eqref{eq9} and \eqref{eq10} into \eqref{eq8}, we obtain
\begin{eqnarray*}
|\Phi(y)-\Phi(y')|\leq \left[2TL\left(1+\sqrt{2}\frac{K}{k}\left(1+\sqrt{T}\right)\right)+L\sqrt{2T}\right]\|y-y'\|_{\infty}.
\end{eqnarray*}
Since we assume that $2TL\left(1+\sqrt{2}\frac{K}{k}\left(1+\sqrt{T}\right)\right)+L\sqrt{2T}< 1$, it is not difficult to see that $\Phi$ is a contraction.

Let us denote by $Y\in \mathcal{S}^{\infty}$ the unique fixed point and by $A$ the associating reflecting process defined by $A_t=\sup_{0\leq v\leq t^+}L_{v}$, where $L$ satisfies  the representation
\begin{eqnarray}
X_{t}=\E\left\{\xi+\int_{t}^{T}f\left(s,Y_s,\sup_{t\leq v\leq s}L_v\right)ds+\int_{t}^{T}g\left(s,Y_s\right)dB_u|\mathcal{F}_t\right\}.\label{eq11'}
\end{eqnarray}
Let us now prove that $(Y,A)$ is the solution to the VRBDSDE \eqref{eq11}. For this instance, it follows from \eqref{eq11'}, the definition of $A$, and the monotonicity of $f$ on the third variable that for all $t\in [0,T]$,
\begin{eqnarray*}
Y_{t}&=&\E\left\{\xi+\int_{t}^{T}f\left(s,Y_s,A_s\right)ds+\int_{t}^{T}g\left(s,Y_s\right)dB_s|\mathcal{F}_t\right\}\nonumber\\
&\leq & \E\left\{\xi+\int_{t}^{T}f\left(s,Y_s,\sup_{t\leq v\leq s}L_v\right)ds+\int_{t}^{T}g\left(s,Y_s\right)dB_s|\mathcal{F}_s\right\}=X_t.
\end{eqnarray*}
To end the proof of existence, it remains to show that the flat-off conditions holds. The properties of optional projection and definition of $A$ and $L$ lead to
\begin{eqnarray*}
\E\int_0^T(Y_t-X_t)dA_t&=&\E\int_0^T\left\{\int_{t}^{T}\left[f\left(s,Y_s,\sup_{0\leq v\leq s^+}L_v\right)-f\left(s,Y_s,\sup_{t\leq v\leq s}L_v\right)\right]ds\right\}dA_t.
\end{eqnarray*}
Next, using the Fubini theorem and the fact that Lebesgue measure does not charge the discontinuities of the path $u\mapsto\sup_{t\leq v\leq u}L_v$, which are only countably many, we have
\begin{eqnarray*}
\E\int_0^T(Y_t-X_t)dA_t&=&\E\int_0^T\left\{\int_{0}^{s}\left[f\left(s,Y_s,\sup_{0\leq v\leq s^+}L_v\right)-f\left(s,Y_s,\sup_{t\leq v\leq s^+}L_v\right)\right]dA_t\right\}ds,\nonumber\\\label{eq12}
\end{eqnarray*}
which provide by the same argument used in \cite{MW} that
\begin{eqnarray*}
\E\int_0^T|Y_t-X_t|dA_t=0.
\end{eqnarray*}
For the uniqueness, let us suppose that there is another solution $(Y',A')$ to the VRBDSDE such that $Y'_t\leq X_t,\; t\in[0,T]$, and
\begin{eqnarray*}
Y'_{t}=\E\left\{\xi+\int_{t}^{T}f\left(s,Y'_s,A'_s\right)ds+\int_{t}^{T}g\left(s,Y'_s\right)dB_s|\mathcal{F}_t\right\},\;\;\;\E\int_0^T|Y'_t-X_t|dA_t=0.
\end{eqnarray*}
Since both $Y$ and $Y'$ are the unique fixed points of the mapping $\Phi$, it follows that $Y=Y'$. Let us consider the stochastic variant Skorohod problem
\begin{eqnarray}
&&\widetilde{Y}_{t}=\E\left\{\xi+\int_{t}^{T}f^Y\left(s,\widetilde{A}_s\right)ds+\int_{t}^{T}g^Y\left(s\right)dB_s|\mathcal{F}_t\right\},\nonumber\\
&&\widetilde{Y}_t\leq X_t,\;\;\;\;\; Y_T=X_T=\xi,\label{eq13}\\
&&\E\int_0^T|\widetilde{Y}_t-X_t|d\widetilde{A}_t=0,\nonumber
\end{eqnarray}
where $f^Y\left(s,l\right)=f\left(s,Y_s,l\right)$ and $g^Y\left(s\right)=g\left(s,Y_s\right)$. Thanks to Theorem \ref{T2}, there exists a unique pair of process
$(\widetilde{Y},\widetilde{A})$ that solves the stochastic variant Skorohod problem. Moreover, since $(Y,A)$ and $(Y',A')$ are the solutions to the variant BDSDE
\eqref{eq13}, it follows that $Y_t=\widetilde{Y}_t$ and $A_t=\widetilde{A}_t=A'_t,\; t\in [0,T]$, a.s., which proves the uniqueness, whence the theorem.
\end{proof}
\begin{corollary}
Suppose that $(Y,A)$ is a solution to VRBDSDE with generator $f$ and $g$ and upper boundary $X$. Then $A_{0^-}=-\infty$ and $Y_0=X_0$.
\end{corollary}
\begin{proof}
Since the existence and uniqueness proof depends heavily on the well-posedness result of the extended stochastic representation theorem, we must require that
$A_{0^-}=-\infty$. On the other hand, since $Y$ is a fixed point of the mapping $\Phi$ defined by \eqref{eq2}, it not difficult to see that $Y_0$ and $X_0$
satisfy the following equalities:
\begin{eqnarray*}
X_{0}&=&\E\left\{\xi+\int_{0}^{T}f\left(s,Y_s,\sup_{0\leq v\leq s}L_v\right)ds+\int_{0}^{T}g\left(s,Y_s\right)dB_s\right\},\\
Y_{0}&=&\E\left\{\xi+\int_{0}^{T}f\left(s,Y_s,A_s\right)ds+\int_{0}^{T}g\left(s,Y_s\right)dB_s\right\}\\
&=&\E\left\{\xi+\int_{0}^{T}f\left(s,Y_s,\sup_{0\leq v\leq s^+}L_v\right)ds+\int_{0}^{T}g\left(s,Y_s\right)dB_s\right\}.
\end{eqnarray*}
Hence, by the same argument that the paths of the increasing process $u\mapsto\sup_{t\leq v\leq u}L_v$ has only countably many discontinuities, which are negligible
under the Lebesgue measure, we prove that $Y_0=X_0$.
\end{proof}
\subsection{Comparison theorems}
This section is devoted to study the comparison theorem of the VRBDSDE, one of the very important tools in the theory of BSDEs. Let us remark that our
method follows closely to one appeared in \cite{MW}, which is quite different from all the existing arguments in the BSDE literature.

To state, let us consider the following two VRBDSDEs for $i=1,2$,
\begin{eqnarray}
&&Y^i_{t}=\E\left\{\xi^i+\int_{t}^{T}f^i\left(s,Y^i_s,A^i_s\right)ds+\int_{t}^{T}g\left(s,Y^i_s\right)dB_s|\mathcal{F}_t\right\},\nonumber\\
&&Y^i_t\leq X^i_t,\;\;\;\;\; Y^i_T=X^i_T=\xi^i,\label{eq14}\\
&&\E\int_0^T|Y^i_t-X^i_t|dA^i_t=0.\nonumber
\end{eqnarray}
In the sequel, we call $(f^i,g,X^i),\; i=1,2$ as the "parameters" of the VRBDSDE \eqref{eq14}, $i=1,2$, respectively. We also define the two following stopping times
\begin{eqnarray}
\mu&=&\inf\left\{t\in[0,T),\; A_t^2>A^1_t+\varepsilon\right\}\wedge T;\nonumber\\
\tau&=&\inf\left\{t\in[\mu,T),\; A_t^1>A^2_t-\frac{\varepsilon}{2}\right\}\wedge T.\label{eq15}
\end{eqnarray}
We recall the following result appear in \cite{MW}.
\begin{lemma}\label{L321}
The stopping times $\mu$ and $\tau$ defined by \eqref{eq15} have the standing properties:
\begin{description}
\item [\rm (i)] $\; \mu$ and $\tau$ are points of increase for $A^2$ and $A^1$, respectively. In other word, for any $\delta>0$, it holds that $A^2_{\mu^-}<A^2_{\mu+\delta}$ and
$A^1_{\tau^-}<A^1_{\tau+\delta}$.
\item [\rm (ii)] $\; \P(\mu<\tau)=1$, and $A^1_{t}\leq A^2_{t}-\frac{\varepsilon}{2}$, for all $t\in[\mu,\tau],\, \P$-a.s.,
\item [\rm (iiii)]  it holds that $Y^2_{\mu}=X^2_{\mu}$ and $Y^1_{\tau}=X^1_{\tau},\; \P$-a.s.
\end{description}
\end{lemma}
Before give the comparison theorem, in order to simplify the notations, let us give the following. For $(Y^i,A^i),\; i=1,2$ be the solution to two VRBDSDEs with boundaries
$X^1$ and $X^2$ respectively, we denote $\Delta\Theta=\Theta^1-\Theta^2,\; \Theta=X,Y,A$, and $\xi$. Furthermore, recall
\begin{eqnarray*}
\mathcal{G}_t=\mathcal{F}^{W}_t\vee\mathcal{F}^{B}_T,
\end{eqnarray*}
we define two martingales
\begin{eqnarray*}
M^i_t=\E\left\{\int_{0}^{T}f^i\left(s,Y^i_s,A^i_s\right)ds+\int_{0}^{T}g\left(s,Y^i_s\right)dB_s|\mathcal{G}_t\right\},\; t\in[0,T],\, i=1,2.
\end{eqnarray*}
\begin{theorem}\label{T321}
Assume that the parameters of the VRBDSDEs \eqref{eq14} $(f^i,g,X^i),\, i=1,2$, satisfy $({\bf A1})$ and $({\bf A2})$. Assume further that
\begin{description}
\item [\rm(i)] $\; f^1(t,y,l)\geq f^2(t,y,l),\; d\P\otimes dt$ a.s.,
\item [\rm(ii)] $ \; X^1_t\leq X^2_t,\;\; 0\leq t\leq T$, a.s.,
\item [\rm(iii)] $\; \Delta X_s\leq \E\{{\rm e}^{(L+\frac{1}{2}L^2)(t-s)}\Delta X_t|\mathcal{G}_s\}$ a.s. for all $s$ and $t$ such that $s<t$.
\end{description}
Then, we have $A^1_t\geq A^2_t,\; t\in[0,T],\ \P$-a.s.
\end{theorem}
\begin{remark}
As it is explained in \cite{MW}, the assumption $(iii)$ in above theorem significate that the process ${\rm e}^{Ls}\Delta X_s$, is a submartingale and it does not add restrictive
on the regularity of the boundary processes $X^1$ and $X^2$, which are only required to be the optional processes satisfying $({\bf A3})$.
\end{remark}

\begin{proof}[Proof of Theorem\, \ref{T321}]
According to \eqref{eq14} and the previous notations, we can write, on the set $\{\mu<T\}$
\begin{eqnarray}
\Delta Y_{\mu}&=&\E\left\{\Delta Y_{\tau}+\int_{\mu}^{\tau}\left[f^1\left(s,Y^1_s,A^1_s\right)-f^2\left(s,Y^2_s,A^2_s\right)\right]ds\right.\nonumber\\
&&+\left.\int_{\mu}^{\tau}\left[g\left(s,Y^1_s\right)-g\left(s,Y^2_s\right)\right]dB_s+(\Delta M_{\tau}-\Delta M_{\mu})|\mathcal{F}_{\mu}\right\},\label{eq16}
\end{eqnarray}
where $\Delta M=M^1-M^2$, and
\begin{eqnarray*}
\nabla_{y}f^1_{s}&=&\frac{f^1\left(s,Y^1_s,A^1_s\right)-f^1\left(s,Y^2_s,A^1_s\right)}{Y^1_s-Y^2_s}{\bf 1}_{\{Y^1_s\neq Y^2_s\}},\\
\nabla_{y}g_{s}&=&\frac{g\left(s,Y^1_s\right)-g\left(s,Y^2_s\right)}{Y^1_s-Y^2_s}{\bf 1}_{\{Y^1_s\neq Y^2_s\}},\\
\Delta_{l}f^1_s&=& f^1\left(s,Y^2_s,A^1_s\right)-f^2\left(s,Y^2_s,A^2_s\right),\\
\Delta_{2}f_s&=& f^1\left(s,Y^2_s,A^2_s\right)-f^2\left(s,Y^2_s,A^2_s\right).
\end{eqnarray*}
It is clear that $({\bf A2})$ implies that $\nabla_{y}f^1$ and $\nabla_{y}g$ are bounded progressively measurable processes, and by the definition of
$\mu,\, \tau$ and the monotonicity of $f$ on it variable $l$, we have $\Delta_{l}f^1>0$ on the interval $[\mu,\tau]$.
Hence, $\Delta Y$ is a unique
solution of the following linear BDSDE
\begin{eqnarray*}
\Delta Y_{\mu}&=&\E\left\{\Delta Y_{\tau}+\int_{\mu}^{\tau}\nabla_{y}f^1_{s}\Delta Y_sds+\int_{\mu}^{\tau}[\Delta_{l}f^1_s+\Delta_{2}f_s]ds\right.\nonumber\\
&&+\left.\int_{\mu}^{\tau}\nabla_{y}g_{s}\Delta Y_sdB_s+(\Delta M_{\tau}-\Delta M_{\mu})|\mathcal{F}_{\mu}\right\}\label{eq16}
\end{eqnarray*}
Setting $\Gamma_t=\exp\left(\int_{0}^{t}\nabla_{y}f^1_{s}ds+\int_{0}^{t}\nabla_{y}g_{s}dB_s-\frac{1}{2}\int_{0}^{t}|\nabla_{y}g_{s}|^2ds\right)$,
as it done in \cite{AON}, one can derive
\begin{eqnarray*}
\E\left\{\Gamma_{\mu}\Delta Y_{\mu}-\Gamma_{\tau}\Delta Y_{\tau}|\mathcal{F}_{\mu}\right\}=\E\left\{\int_{\mu}^{\tau}\Gamma_s[\Delta_{l}f^1_s+\Delta_{2}f_s]ds
-\int_{\mu}^{\tau}\Gamma_{s}d(\Delta M_s)|\mathcal{F}_{\mu}\right\}.
\end{eqnarray*}
Therefore, since $f^1\geq f^2,\; \Delta_{2}f\geq 0, \, d\P\otimes dt$-a.s., and consequently, since $M^i,\; i=1,2$, is a martingale, we get
\begin{eqnarray}
\E\left\{\Gamma_{\mu}\Delta Y_{\mu}-\Gamma_{\tau}\Delta Y_{\tau}|\mathcal{F}_{\mu}\right\}=\E\left\{\int_{\mu}^{\tau}\Gamma_s[\Delta_{l}f^1_s
+\Delta_{2}f_s]ds|\mathcal{F}_{\mu}\right\}>0.\label{eq17}
\end{eqnarray}
On the other hand, by the flat-off condition and Lemma \ref{L321}-$(iii)$, one can check that $Y^1_{\mu}-Y^2_{\mu}\leq X^1_{\mu}-X^2_{\mu}$ and
$Y^1_{\tau}-Y^2_{\tau}\leq X^1_{\tau}-X^2_{\tau}$,
\begin{eqnarray}
\E\left\{\Gamma_{\mu}\Delta Y_{\mu}-\Gamma_{\tau}\Delta Y_{\tau}|\mathcal{F}_{\mu}\right\}
\leq \E\left\{\Gamma_{\mu}\Delta X_{\mu}-\Gamma_{\tau}\Delta X_{\tau}|\mathcal{F}_{\mu}\right\}.\label{eq18}
\end{eqnarray}
It is now clear that if the right hand side of \eqref{eq18} is non-positive, then \eqref{eq18} contradicts to \eqref{eq17}, and therefore one must have $\P(\mu<T)=0$. In other
words, $A^2_t\leq A^1_t+\varepsilon$, for all $t\in[0,T],\, \P$-a.s. Since $\varepsilon$ is taken arbitrary, entails that
\begin{eqnarray*}
 A^2_t\leq A^1_t,\;\;\;\; t\in [0,T],\;\; \P\mbox{-a.s.}
\end{eqnarray*}
Now it remain to show that the right hand side of \eqref{eq18} is non-positive. To do this, let us note that since by assumption $(ii)$ we have $\Delta X_{\tau}\leq 0$,
it follows from \eqref{eq18} and assmption $(iii)$ that
\begin{eqnarray*}
\E\left\{\Gamma_{\mu}\Delta Y_{\mu}-\Gamma_{\tau}\Delta Y_{\tau}|\mathcal{F}_{\mu}\right\}
&\leq& \Gamma_{\mu}\E\left\{\Delta X_{\mu}-{\rm e}^{\int_{\mu}^{\tau}\nabla_{y}f^1_{s}ds
+\int_{\mu}^{\tau}\nabla_{y}g_{s}dB_s-\frac{1}{2}\int_{\mu}^{\tau}|\nabla_{y}g_{s}|^2ds}\Delta X_{\tau}|\mathcal{F}_{\mu}\right\}\\
&\leq & \Gamma_{\mu}\E\left\{\Delta X_{\mu}-{\rm e}^{(L+\frac{1}{2}L^2)(\tau-\mu)}\Delta X_{\tau}|\mathcal{F}_{\mu}\right\}\\
&\leq & 0.
\end{eqnarray*}
\end{proof}
As it is emphasized in \cite{MW}, Theorem \ref{T321} only gives the comparison between the two reflecting processes $A^1$ and $A^2$. This is still one step away from
comparison between $Y^1$ and $Y^2$, which is much desirable for obvious reason. But, the latter is not true in general, due do the "opposite" monotonicity on $f^{i}$'s
on the variable $l$. We nevertheless have the following corollary of Theorem \ref{T321}.
\begin{corollary}
Assume all the assumptions of Theorem \ref{T321} hold and further $f^1=f^2$. Then $Y^1_t\leq Y^2_t$, for all $t\in[0,T],\, \P$-a.s.
\end{corollary}
\begin{proof}
Let us denote $f=f^1=f^2$ and define two random functions $\widetilde{f}^i(t,\omega,y)=f(t,\omega,y,A^i_t(\omega))$, for $(t,\omega,y)\in[0,T]\times\Omega\times\R, i=1,2$.
Then $Y^1$ and $Y^2$ can be seen as the solution of BDSDEs
\begin{eqnarray*}
Y^i_{t}=\E\left\{\xi^i+\int_{t}^{T}\widetilde{f}^i\left(s,Y^i_s\right)ds+\int_{t}^{T}g\left(s,Y^i_s\right)dB_s|\mathcal{F}_t\right\},\; t\in[0,T],\;\; i=1,2.
\end{eqnarray*}
It follows from the fact $A^1\geq A^2$ that $\widetilde{f}^1(t,\omega,y)=f(t,\omega,y,A^1_t(\omega))\leq f(t,\omega,y,A^2_t(\omega))=\widetilde{f}^2(t,\omega,y)$.
 Therefore, since $\xi^1=X^1_T\leq X^2_T=\xi^2$, and according to the comparison theorem of BDSDEs, we have $Y^1_t\leq Y^2_t$, for all $t\in[0,T],\ \P$-a.s.
\end{proof}
\subsection{Stability results}
In this section, we study another useful aspect of the well-posedness of the VRBDSDE, which it is called the continuous dependence of the solution on the boundary process
whence the terminal process as well.
For this instance, let us introduce, for any optional process $X$ and any stopping time $\mu$ and $\tau$ satisfy that $\mu<\tau$,
\begin{eqnarray*}
m_{\mu,\tau}(X)=\frac{\E\{X_{\tau}-X_{\mu}|\mathcal{F}_{\mu}\}}{\E\{\tau-\mu|\mathcal{F}_{\mu}\}}.
\end{eqnarray*}
Let us note that the random variable $m_{\mu,\tau}(X)$ measures the path regularity of the "nonmartingale" part of the boundary process $X$.
In the sequel, we will show that this will be a major measurement for the "closeness" of the boundary processes, as far as the continuous dependence is concerned.

Let us consider $\{X^n\}_{n=1}^{\infty}$, a sequence of optional processes satisfying that $({\bf A3})$. We suppose that $\{X^n\}_{n=1}^{\infty}$ converges
to $X^0$ in $\mathcal{S}^{\infty}$, and that $X^0$ satisfies $({\bf A3})$ as well. Let $(Y^n,A^n)$ be the unique solution to the VRBDSDE's with parameters $(f,g,X^n)$, for
$n=0,1,2,\cdot\cdot\cdot\cdot\cdot$. Roughly speaking, for $n=0,1,2,\cdot\cdot\cdot\cdot$, we have
\begin{eqnarray*}
&&X^n_{t}=\E\left\{\xi^n+\int_{t}^{T}f\left(s,Y^n_s,\sup_{t\leq v\leq s}L^n_v\right)ds+\int_{t}^{T}g\left(s,Y^n_s\right)dB_s|\mathcal{F}_t\right\},\\
&&A^n_s=\sup_{0\leq v\leq s^+}L^n_v,\\
&&Y^n_{t}=\E\left\{\xi^n+\int_{t}^{T}f\left(s,Y^n_s,A^n_s\right)ds+\int_{t}^{T}g\left(s,Y^n_s\right)dB_s|\mathcal{F}_t\right\}.
\end{eqnarray*}

Next, let us give the following lemma that provides the control of $|A^n_t-A_t^0|$, which is needed in the sequel.
\begin{lemma}\label{L331}
Assume $({\bf A2})$ and $({\bf A3})$ hold. Then, for all $t\in[0,T]$, it holds that
\begin{eqnarray*}
|A^n_t-A_t^0|\leq \frac{\sqrt{3}}{k}\sup_{0\leq s\leq t}\mbox{ess}\sup_{\tau>s}\left|m^n_{\mu,\tau}-m^0_{\mu,\tau}\right|+\frac{\sqrt{3}L}{k}(1+\sqrt{T})\|Y^n-Y^0\|_{\infty}.
\end{eqnarray*}
\end{lemma}
\begin{proof}
The proof follows the similar step as the proof of Lemma \ref{L312}. Let us consider, $l^n_{s,\tau},\; n=0,1,\cdot\cdot\cdot$ the $\mathcal{F}_s$-measurable
random variable such that
\begin{eqnarray}
\E(X^n_s-X^n_{\tau}|\mathcal{F}_{s})&=&\E\left\{\int_{s}^{\tau}f\left(u,Y^n_u,l^n_{s,\tau}\right)du+\int_{s}^{\tau}g\left(u,Y^n_u\right)dB_u|\mathcal{F}_s\right\}.
\end{eqnarray}
Therefore, for $n=1,\cdot\cdot\cdot$, we have
\begin{eqnarray*}
\E(X^n_s-X^n_{\tau}|\mathcal{F}_{s})-\E(X^0_s-X^0_{\tau}|\mathcal{F}_{s})
&=&\E\left\{\int_{s}^{\tau}[f\left(u,Y^n_u,l^n_{s,\tau}\right)-f\left(u,Y^0_u,l^0_{s,\tau}\right)]du\right.\\
&&\left.+\int_{s}^{\tau}[g\left(u,Y^n_u\right)-g\left(u,Y^0_u\right)]dB_u|\mathcal{F}_s\right\}.
\end{eqnarray*}
On the set $D^{\tau}_{s}=\left\{\omega/ l^n_{s,\tau}(\omega)>l^0_{s,\tau}(\omega)\right\}\in \mathcal{F}_s$, we get
\begin{eqnarray*}
&&{\bf 1}_{D^{\tau}_{s}}\E(X^n_s-X^n_{\tau}|\mathcal{F}_{s})-\E(X^0_s-X^0_{\tau}|\mathcal{F}_{s})\\
&=&{\bf 1}_{D^{\tau}_{s}}\E\left\{\int_{s}^{\tau}[f\left(u,Y^n_u,l^n_{s,\tau}\right)-f\left(u,Y^0_u,l^n_{s,\tau}\right)+f\left(u,Y^0_u,l^n_{s,\tau}\right)
-f\left(u,Y^0_u,l^0_{s,\tau}\right)]du\right.\\
&&\left.+\int_{s}^{\tau}[g\left(u,Y^n_u\right)-g\left(u,Y^0_u\right)]dB_u|\mathcal{F}_s\right\}.
\end{eqnarray*}
From $({\bf A2})$, it clear that on $D^{\tau}_{s},\, f\left(u,Y^0_u,l^n_{s,\tau}\right)
-f\left(u,Y^0_u,l^0_{s,\tau}\right)\geq k|l^n_{s,\tau}-l^0_{s,\tau}|$ and hence
\begin{eqnarray*}
k^2[|l^n_{s,\tau}-l^0_{s,\tau}|\E\left\{\tau-s|\mathcal{F}_s\right\}]^2{\bf 1}_{D^{\tau}_{s}}
&\leq &3|\E(X^n_s-X^n_{\tau}|\mathcal{F}_{s})-\E(X^0_s-X^0_{\tau}|\mathcal{F}_{s})|^2{\bf 1}_{D^{\tau}_{s}}\\
&&+
3\left|\E\left\{\int_{s}^{\tau}L|Y^n_u-Y^0_u|du\mathcal{F}_s\right\}\right|^2{\bf 1}_{D^{\tau}_{s}}\\
&&\left.+3\left|\int_{s}^{\tau}[g\left(u,Y^n_u\right)-g\left(u,Y^0_u\right)]dB_u|\mathcal{F}_s\right\}\right|^2{\bf 1}_{D^{\tau}_{s}}.
\end{eqnarray*}
Next, assumption $({\bf A2})$-$(iii)$ together with conditional expectation version of isometry \newline property lead to
\begin{eqnarray*}
|l_{s,\tau}-l'_{s,\tau}|\leq \frac{\sqrt{3}}{k}\left|m^n_{\mu,\tau}-m^0_{\mu,\tau}\right|+\frac{\sqrt{3}\,L}{k}
\left(1+[\E\left\{(\tau-s)|\mathcal{F}_s\right\}]^{-1/2}\right)\|y-y'\|_{\infty}
\end{eqnarray*}
on $D^{\tau}_{s}$. Similarly, we can show that the inequality holds on the complement of $D^{\tau}_{s}$ as well. Therefore, we have
\begin{eqnarray*}
|l_{s,\tau}-l'_{s,\tau}|\leq \frac{\sqrt{3}}{k}\left|m^n_{\mu,\tau}-m^0_{\mu,\tau}\right|+
\frac{\sqrt{3}\,L}{k}\left(1+[\E\left\{(\tau-s)|\mathcal{F}_s\right\}]^{-1/2}\right)\|y-y'\|_{\infty}
\end{eqnarray*}
Finally, according to the definition of $A^n,\, n=0,1,\cdot\cdot\cdot$, we conclude that for $n=1,2,\cdot\cdot\cdot$,
\begin{eqnarray*}
|A_t-A'_t|&=&\left|\sup_{0\leq s\leq t}L^n_{s}-\sup_{0\leq s\leq t}L^0_{s}\right|\\
&\leq& \sup_{0\leq s\leq t}\left|\mbox{ess}\inf_{\tau>s}l^n_{s,\tau}-\mbox{ess}\inf_{\tau>s}l^0_{s,\tau}\right|\\
&\leq&\sup_{0\leq s\leq t}\mbox{ess}\sup_{\tau>s}|l^n_{s,\tau}-l^0_{s,\tau}|\\
&\leq &\sup_{0\leq s\leq t}\mbox{ess}\sup_{\tau>s}\frac{\sqrt{3}}{k}\left[\left|m^n_{\mu,\tau}-m^0_{\mu,\tau}\right|
+L\left(1+\left(\E\left\{(\tau-s)|\mathcal{F}_s\right\}\right)^{-1/2}\right)\right]\|Y^n-Y^0\|_{\infty}\\
&\leq & \frac{\sqrt{3}}{k}\sup_{0\leq s\leq t}\mbox{ess}\sup_{\tau>s}\left|m^n_{\mu,\tau}-m^0_{\mu,\tau}\right|+\frac{\sqrt{3}L}{k}(1+\sqrt{T})\|Y^n-Y^0\|_{\infty}.
\end{eqnarray*}
\end{proof}
Now, we are ready to derive the main result of this subsection.
\begin{theorem}\label{T331}
Assume $({\bf A2})$ and $({\bf A3})$ hold. Further, assume that $${\sqrt{6}TL\left(1+\sqrt{3}\frac{K}{k}\left(1+\sqrt{T}\right)\right)+L\sqrt{3T}< 1}.$$ Then, it holds that
\begin{eqnarray*}
\|Y^n-Y^0\|_{\infty}&\leq & \frac{\sqrt{3}}{1-\left[\sqrt{6}TL\left(1+\sqrt{3}\frac{K}{k}\left(1+\sqrt{T}\right)\right)+L\sqrt{3T}\right]}\times \\ &&\times\left\{ \left\|\xi^n-\xi^0\right\|_{\infty}
+\frac{\sqrt{6}TK}{k}\left\|\sup_{\mu\in[0,T]}ess\sup_{\tau>\mu}\frac{1}{k}\left|m^n_{\mu,\tau}-m^0_{\mu,\tau}\right|\right\|_{\infty}\right\}.
\end{eqnarray*}
\end{theorem}
\begin{proof}
Using the similar arguments as Theorem \ref{T311}, we obtain this estimation
\begin{eqnarray*}
|Y^n_t-Y^{0}_t|\leq \sqrt{3}\|\xi^n-\xi^0\|_{\infty}+\sqrt{3}L(\sqrt{2}T+\sqrt{T})\|Y^n-Y^0\|_{\infty}+\sqrt{6}TK\|A^n-A^0\|_{\infty},\label{eq19}
\end{eqnarray*}
which, together with Lemma \ref{L331}, proves the desired result.
\end{proof}

\begin{remark}
Let us emphasize that, since all the model study in section 6 of \cite{MW} have their stochastic counterpart, we can with no more difficulty establish respectively the stochastic version of recursive intertemporal utility minimization, optimal stopping problems. It suffice to follows the similar step as in \cite{MW} with some additional argument due to the presence of the backward stochastic integral with respect the Brownian motion $B$.
\end{remark}

{\bf\large Acknowledgement} The work was partially done while the second author visited
Shandong University. He would like to thank Prof. Shige Peng for providing a stimulating
working environment.

\setcounter{theorem}{0} \setcounter{equation}{0}

\label{lastpage-01}
\end{document}